\documentclass{llncs}
\usepackage{amsmath,amssymb,amsfonts,tikz,ifthen}
\usepackage[T1]{fontenc}
\usetikzlibrary{patterns}
\newcounter{observation}[section]
\newtheorem{obs}[observation]{Observation}
\newcommand{\fmod}{\ensuremath{\mathrm{mod}}}
\begin{document}
\title{A $p$-centered coloring for the grid using $O(p)$ colors}
\author{Mathew Francis\inst{1} \and Drimit Pattanayak\inst{2}}
\institute{Indian Statistical Institute, Chennai Centre. Email: \texttt{mathew@isichennai.res.in} \and Indian Statistical Institute, Kolkata. Email: \texttt{drimitpattanayak@gmail.com}}
\maketitle
\pagestyle{plain}
\begin{abstract}
A $p$-centered coloring of a graph $G$, where $p$ is a positive integer, is a coloring of the vertices of $G$ in such a way that every connected subgraph of $G$ either contains a vertex with a unique color or contains more than $p$ different colors. 
We give an explicit construction of a $p$-centered coloring using $O(p)$ colors for the planar grid. 
\end{abstract}
\section{Introduction}

We denote the vertex set and edge set of a graph $G$ by $V(G)$ and $E(G)$ respectively.
For a graph $G$ and a set $S \subseteq V(G)$ we denote by $G[S]$ the subgraph induced in $G$ by $S$, i.e. $V(G[S])=S$ and $E(G[S])=\{uv: u,v \in S, \,\, uv\in E(G)\}$.
Please refer to~\cite{diestel} for any term that is not defined here.

For a graph $G$, a \emph{coloring} of the vertices of $G$ is a mapping $\phi:V(G)\rightarrow Q$, where $Q$ is any set. Given such a coloring $\phi$, we say that the set $Q$ is the set of \emph{colors} used by $\phi$ and that a vertex $u\in V(G)$ has the color $\phi(u)$. Given a graph $G$ and a positive integer $p$, a \emph{$p$-centered coloring} of $G$ is a coloring of the vertices of $G$ using a set of colors $Q$ such that for every connected subgraph $H$ of $G$, either there exists $i\in Q$ such that $|\phi^{-1}(i)\cap V(H)|=1$ or $|\phi(V(H))|>p$, or in other words, every connected subgraph of $G$ either contains a vertex with a unique color or contains more than $p$ colors.
\medskip

The notion of $p$-centered colorings of graphs was introduced by Ne\v{s}et\v{r}il and Ossona de Mendez~\cite{NesPOMtreedepth} in the course of their development of the theory of ``structurally sparse'' graph classes. In 2004, DeVos et al.~\cite{DeVosetal} showed that every proper minor closed class of graphs has a \emph{low treewidth coloring}: there exists a function $f:\mathbb{N}\rightarrow\mathbb{N}$ such that the vertex set of any graph in the class can be partitioned into $f(p)$ parts in such a way that for every $i\leq p$, the union of any $i$ parts induces a subgraph of treewidth at most $i-1$. Ne\v{s}et\v{r}il and Ossona de Mendez~\cite{NesPOMtreedepth} introduced the \emph{treedepth} of a graph and showed that it is exactly the minimum number of colors needed in a \emph{centered coloring} of the graph: a coloring of the vertices such that every connected subgraph contains a vertex having a unique color (this parameter also turns out to be equal to some other graph parameters that were known in the literature; see~\cite{NesPOMtreedepth} for details). It is not difficult to see that a $p$-centered coloring of a graph $G$ is a coloring of $V(G)$ such that for every $i\leq p$, the union of $i$ color classes induces a subgraph of $G$ whose each connected component has a centered coloring using at most $i$ colors. Thus, a $p$-centered coloring can be said to be a ``low treedepth coloring''. Since the treedepth of a graph is always at least one more than its treewidth~\cite{Bodlaenderetal} (in fact, one more than its pathwidth) and the treewidth of any graph is equal to the maximum of the treewidth of its connected components, low treedepth colorings are a generalization of low treewidth colorings. The following generalization of the result of DeVos et al. was shown in~\cite{NesPOMtreedepth}: for every proper minor closed class of graphs, there exists a function $f:\mathbb{N}\rightarrow\mathbb{N}$ such that for every integer $p\geq 1$, every graph in the class has a $p$-centered coloring using at most $f(p)$ colors. In a landmark paper, Ne\v{s}et\v{r}il and Ossona de Mendez~\cite{NesPOMexpansion} improved this further and showed that the classes of graphs having low treedepth colorings, the classes of graphs having low treewidth colorings, the classes of graphs for which there exists a function $f:\mathbb{N}\rightarrow\mathbb{N}$ such that every graph in the class has a $p$-centered coloring using at most $f(p)$ colors for every $p\geq 1$, are all the same and characterized these classes as the classes of graphs with \emph{bounded expansion} (see~\cite{NesPOMexpansion} for details).

It was first shown by Pilipczuk and Siebertz~\cite{PiliSieb} that for every proper minor closed family of graphs, there exists a \emph{polynomial} function $f:\mathbb{N}\rightarrow\mathbb{N}$ such that for every positive integer $p$, every graph in the class has a $p$-centered coloring using at most $f(p)$ colors --- or in other words, proper minor closed families of graphs ``admit polynomial centered colorings''.
D\k{e}bski, Felsner, Micek and Schr\"oder~\cite{DFMS} combined different techniques to obtain several path-breaking results on $p$-centered colorings, including substantial improvements to and tightening of some bounds in~\cite{PiliSieb}. Perhaps the most surprising result in their work is their proof using entropy compression that graphs of bounded degree have $p$-centered colorings using just $O(p)$ colors: they showed that for any positive integer $p$, every graph $G$ having maximum degree $\Delta$ has a $p$-centered coloring using $O(\Delta^{2-\frac{1}{p}}p)$ colors. However, their technique does not provide an explicit construction of a $p$-centered coloring using $O(p)$ colors for graphs of bounded maximum degree, and they remark that an explicit construction of such a coloring is not known even for the planar grid. In this paper, we give a simple and direct construction of a $p$-centered coloring of the planar grid using $O(p)$ colors.

\subsection{Definitions and Notation}
Given a graph $G$ and a coloring $\phi$ of it, a ``violator'' in $G$ with respect to $\phi$ is a connected subgraph that contains at most $p$ colors and also does not contain any vertex with a unique color. Thus, a coloring of a graph is a $p$-centered coloring for it if and only if there is no violator in the graph with respect to that coloring.

For two positive integers $a$ and $b$, we define $\fmod(a,b)=a-b\left\lfloor \frac{a}{b} \right\rfloor$. Also, we write $b\mid a$ to mean ``$b$ divides $a$'', or in other words $\fmod(a,b)=0$, and we write $b\nmid a$ to mean ``$b$ does not divide $a$''. For two integers $a$ and $b$ such that $a\leq b$, we denote by $[a,b]$ the set $\{n\in\mathbb{N}\colon a\leq n\leq b\}$.
We abbreviate $\log_2$ as $\lg$.

For the purposes of this paper, we shall use the term \emph{grid} to denote the (infinite) graph $G$ having $V(G)=\mathbb{N}^2$ and $E(G)=\{(i,j)(i',j'):i,j,i',j' \in \mathbb{N}$ and $|i-i'|+|j-j'|=1\}$. We shall use $G$ to denote the grid graph. As planar grids of all sizes occur as subgraphs of $G$, it suffices to demonstrate a $p$-centered coloring for $G$ using $O(p)$ colors.

For a vertex $(x,y)\in V(G)$, we define $\pi_x(x,y)=x$ and $\pi_y(x,y)=y$.
A set $S\subseteq V(G)$ is said to be \emph{row contiguous} if  $\pi_x(S)=[\min\pi_x(S),\max\pi_x(S)]$ and \emph{column contiguous} if $\pi_y(S)=[\min\pi_y(S),\max\pi_y(S)]$.
In other words, a set $S$ is row contiguous if and only if for $x,a,x'\in\mathbb{N}$ such that $x<a<x'$, we have $x,x'\in\pi_x(S)\Rightarrow a\in\pi_x(S)$, and a set $S$ is column contiguous if and only if for $y,b,y'\in\mathbb{N}$ such that $y<b<y'$, we have $y,y'\in\pi_y(S)\Rightarrow b\in\pi_y(S)$.
We say that a row contiguous set $S$ \emph{spans $|\pi_x(S)|$ rows} and that a column contiguous set $S$ \emph{spans $|\pi_y(S)|$ columns}. Note that if $H$ is any connected subgraph of $G$, then $V(H)$ is both row and column contiguous.

\section{The coloring}
Our proof strategy is more or less as follows:
we first show a coloring of the grid in which every violator is \emph{large}---i.e. every violator spans either ``many'' rows or ``many'' columns, and
then show another coloring of the grid, in which every violator is \emph{small}, or in other words, no violator spans ``many'' rows or columns.
Our final $p$-centered coloring of the grid is constructed by ``mixing'' or ``interleaving'' these two colorings in a special way. The accurate meanings of the words in quotes shall become clear when we give the formal definitions of the colorings.

We show the construction for the case when $p$ is a power of 2 that is greater than 1, or in other words, $p=2^i$, where $i>0$.
\subsection{A coloring in which every violator is large}\label{sec:mu}
We describe a coloring $\mu$ of the grid $G$ that contains no ``small'' violators; more accurately, the coloring will have the property that any violator in $G$ with respect to this coloring spans either more than $4p$ rows or more than $4p$ columns. We actually show a stronger property: any row and column contiguous subset of $V(G)$ contains a uniquely colored vertex or spans either more than $4p$ rows or more than $4p$ columns.
\medskip

\noindent Define the function $f:\mathbb{N}\rightarrow \mathbb{N}$ as $n\mapsto \max\{i\in[0,\lg(4p)]: 2^i\mid n\}$.
\medskip

\noindent We further define $\alpha:\mathbb{N}^2\rightarrow\{0,1\}$, $l:\mathbb{N}^2\rightarrow [0,\lg(4p)]$, and  $\rho:\mathbb{N}^2\rightarrow \mathbb{N}$ as follows.
Let $(x,y)\in\mathbb{N}^2$.
\medskip

If $f(x)\geq f(y)$, then we define $\alpha(x,y)=0$, $l(x,y)=f(x)$, and $\rho(x,y)=\fmod(y,2^{f(x)+1})$.

If $f(x)<f(y)$, then we define $\alpha(x,y)=1$, $l(x,y)=f(y)$, and $\rho(x,y)=\fmod(x,2^{f(y)+1})$.
\medskip

\noindent Notice that for any $(x,y)\in\mathbb{N}^2$, we have $l(x,y)=\max\{f(x),f(y)\}$.
\medskip



\noindent Finally, we define $\mu:\mathbb{N}^2\rightarrow\{0,1\}\times [0,\lg(4p)]\times\mathbb{N}$ as $(x,y)\mapsto(\alpha(x,y),l(x,y),\rho(x,y))$.
We consider $\mu$ to be a coloring of $V(G)$.
\medskip

\noindent Observe that for any $(x,y)\in V(G)$, we have $\rho(x,y)\in [0,2^{l(x,y)+1}-1]$. This immediately implies that $|\mu(V(G))|\leq 2\sum_{i=0}^{\lg(4p)} 2^{i+1}\leq 32p$.\footnote[1]{For a more precise calculation of $|\mu(V(G))|$, note that we can have $\rho(x,y)=0$ only if $l(x,y)=\lg(4p)$ and $\alpha(x,y)=0$. Further, we cannot have $\alpha(x,y)=1$ and $\rho(x,y)=2^{f(y)}$. This can be used to deduce that $$|\mu(V(G))|=2\left(\sum\limits_{i\in [0,\lg(4p)]}(2^{i+1}-1)\right)+1-(\lg(4p)+1)=32p-3\lg(4p)-6=32p-3\lg(p)-12$$}
Figure~\ref{fig:mu} shows how the coloring $\mu$ looks like near the corner $(0,0)$ of the grid.

\begin{figure}
	\begin{center}
		\includegraphics{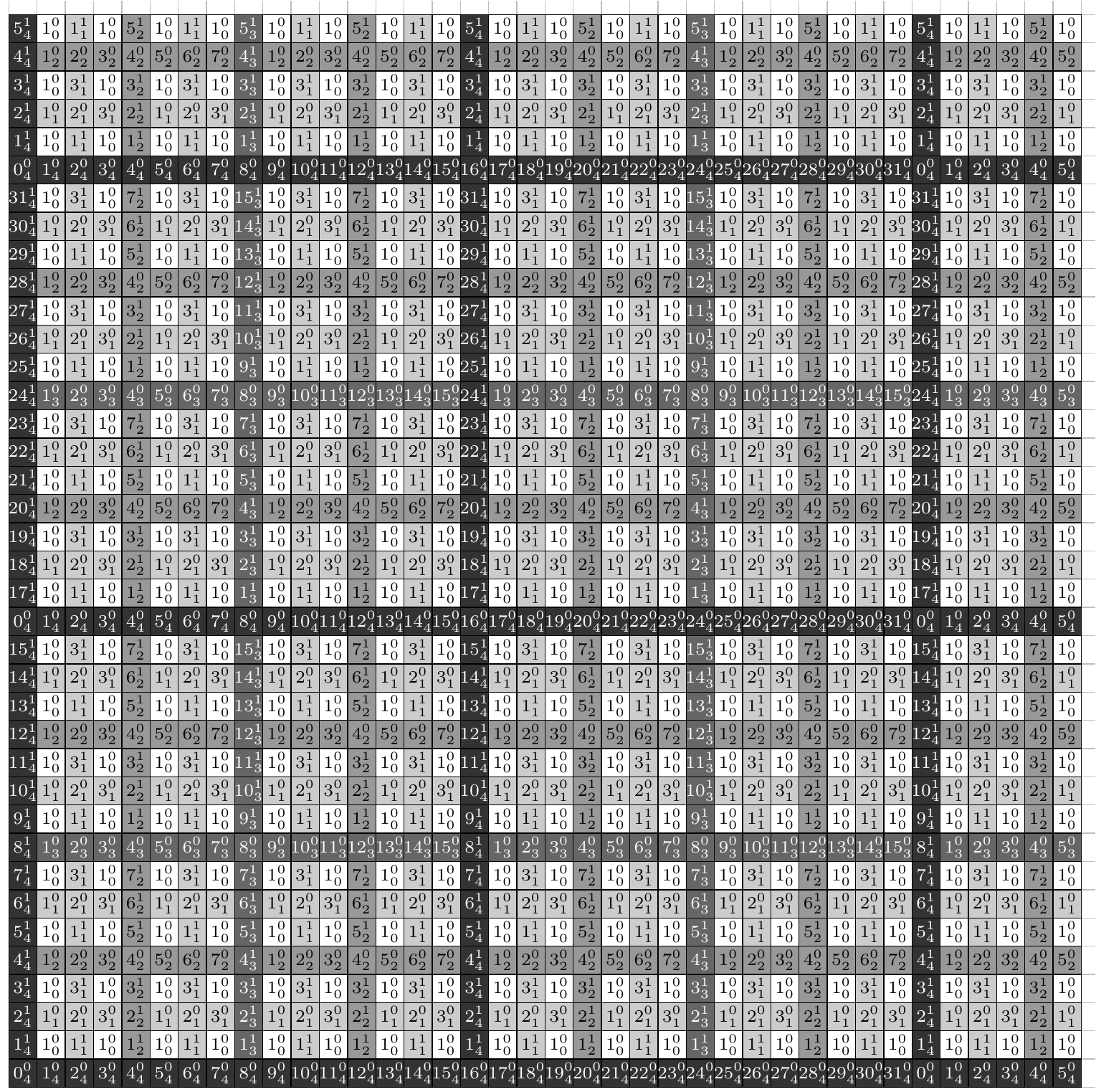}
	\end{center}
	\caption{The coloring $\mu$ of the grid for $p=4$. The color $\mu(i,j)$, in the form $\rho(i,j)_{l(i,j)}^{\alpha(i,j)}$, is written in the position $(i,j)$. The cell in the lower left corner corresponds to the vertex $(0,0)$. The higher the value of $l(i,j)$, the darker is the shade of the cell corresponding to the vertex $(i,j)$.}\label{fig:mu}
\end{figure}

\begin{lemma}\label{small}
Let $S\subseteq V(G)$. If $S$ is both row and column contiguous, and does not contain any uniquely colored vertex with respect to the coloring $\mu$ of $G$, then either $S$ spans more than $4p$ rows or $S$ spans more than $4p$ columns.
\end{lemma}

\begin{proof}
Suppose that $S$ is both row and column contiguous and does not contain any uniquely colored vertex with respect to the coloring $\mu$. Consider a vertex $(x,y)\in S$ such that $l(x,y)=\max\{l(a,b)\colon (a,b)\in S\}$. Since there is no uniquely colored vertex in $S$, there exists $(x',y')\in S$ such that $(x',y')\neq (x,y)$ and $\mu(x,y)=\mu(x',y')$. By definition of $\mu$, we have $\alpha(x,y)=\alpha(x',y')$, $l(x,y)=l(x',y')$, and $\rho(x,y)=\rho(x',y')$. We assume that $\alpha(x,y)=\alpha(x',y')=0$, since the other case, i.e. $\alpha(x,y)=\alpha(x',y')=1$, is symmetric. Then $f(x)\geq f(y)$ and $f(x') \geq f(y')$, from which it follows that $l(x,y)=f(x)$ and $l(x',y')=f(x')$. As $l(x,y)=l(x',y')$, we now have $f(x)=f(x')=t$ (say).

We first consider the case when $t<\lg(4p)$.
Suppose that $x\neq x'$. We shall assume without loss of generality that $x<x'$. Since $f(x)=t<\lg(4p)$, we have that $2^t\mid x$ and that $2^{t+1}\nmid x$. This implies that $2^{t+1}\mid x+2^t$. As $2^t\mid x$ and $2^t\mid x'$, we have $x'\geq x+2^t$. Since $f(x')=t<\lg(4p)$, we have $2^{t+1}\nmid x' $, from which it follows that $x'>x+2^t$. As $S$ is row contiguous, there exists $b\in\mathbb{N}$ such that $(x+2^t,b)\in S$. As $t<\lg(4p)$ and $2^{t+1}\mid x+2^t$, we have that $f(x+2^t)\geq t+1$. But now we have $l(x+2^t,b)=\max\{f(x+2^t),f(b)\}\geq t+1>l(x,y)$, which is a contradiction to our choice of $(x,y)$. So we can assume that $x=x'$. As $(x,y)\neq (x',y')$, we have $y\neq y'$. We assume without loss of generality that $y<y'$. As $\rho(x,y)=\rho(x',y')$, we have $\fmod(y,2^{t+1})=\fmod(y',2^{t+1}) = r$ (say). This implies that $y'\geq y+2^{t+1}$. Since $r<2^{t+1}$, it follows that $y'-r>y$. As $t<\lg(4p)$ and $l(x,y)=\max\{f(x),f(y)\}=t$, we know that $2^{t+1}\nmid y$, or in other words, $r\neq 0$. Thus $y'-r<y'$. We now have $y<y'-r<y'$. As $S$ is column contiguous there exists $a \in \mathbb{N}$ such that $(a,y'-r)\in S$. Notice that $2^{t+1}\mid y'-r$, which means that $f(y'-r)\geq t+1$. Then $l(a,y'-r)=\max\{f(a),f(y'-r)\}\geq t+1>l(x,y)$. But this is again a contradiction to our choice of $(x,y)$.

From the above argument, we can conclude that $t=\lg(4p)$. Hence both $x$ and $x'$ are divisible by $4p$. If $x\neq x'$, then this means that $|x-x'|\geq 4p$, which implies that $S$ spans more than $4p$ rows, and we are done. So we shall assume that $x=x'$, which means that $y\neq y'$. Since $\rho(x,y)=\rho(x',y')$, we have that $\fmod(y,2^{t+1})=\fmod(y',2^{t+1})$, or in other words, $\fmod(y,8p)=\fmod(y',8p)$. Hence $|y-y'|\geq 8p$ (as $y\neq y'$), implying that $S$ spans more than $8p$ columns, and we are again done. This completes the proof.
\hfill\qed
\end{proof}

Notice that the coloring $\mu$ of $G$ is not a $p$-centered coloring of $G$ because of the presence of ``large'' violators. For example, consider the subgraph $H=G[\{(0,1),(1,1),(2,1),\ldots,(8p,1)\}]$. Clearly, for each vertex $u\in V(H)$, we have $\alpha(u)=0$, $\rho(u)=1$, and $l(u)\in [0,\lg(4p)]$, which implies that $|\mu(V(H))|\leq\lg(4p)+1$. So when $p\geq 8$, $|\mu(V(H))|<p+1$. At the same time, we have that for each $x\in [0,4p]$, the vertices $(x,1),(4p+x,1)\in V(H)$ and $\mu(x,1)=\mu(4p+x,1)$. Thus, no vertex in $H$ has a unique color. As $H$ is clearly a connected subgraph of $G$, we can conclude that $\mu$ is not a $p$-centered coloring of $G$ for any value of $p$ greater than or equal to  8.

\subsection{A coloring in which every violator is small}
We now show another coloring of the grid $G$ which is again not $p$-centered, but this time has no ``large violators''; in particular, it has the property that any violator in $G$ spans at most $2(p+1)$ rows and at most $2(p+1)$ columns. Note that we mention this coloring of the grid just to develop intuition; it is not directly used in the proof of the main result.

Let $R=\{(x,y)\in V(G)\colon x+y$ is even$\}$ and let $C=V(G)\setminus R$. Consider the coloring $\theta:\mathbb{N}^2\rightarrow\mathbb{N}$ of the vertices of $G$ that is defined as follows: $$\theta(x,y)=\left\{\begin{array}{l@{\hspace{.25in}}l}(0,\fmod(x,2p+2))&\mbox{if }(x,y)\in R\\(1,\fmod(y,2p+2))&\mbox{if }(x,y)\in C\end{array}\right.$$ Figure~\ref{fig:theta} shows how the coloring looks like for the case $p=3$ for a section of the grid. It is not difficult to prove that if $H$ is any connected subgraph of $G$ that spans at least $2(p+1)$ columns or at least $2(p+1)$ rows, then $|\theta(V(H))|>p$ (we do not prove this rigorously as it is not required for the proof of our main result). Note that $\theta$ is also not a $p$-centered coloring of $G$ as it contains ``small'' violators: for example, for any $p\geq 5$, the connected subgraph $G[\{(0,0),(0,1),(0,2),(1,1),(1,2),(1,3),(2,0),(2,1),(2,2),(3,2)\}]$ of $G$ is a violator.

\begin{figure}
	\begin{center}
		\begin{tikzpicture}
			\foreach \x in {0,...,17} {
				\draw [lightgray] (0,\x*.5) -- (8.75,\x*.5);
				\draw [lightgray] (\x*.5,0) -- (\x*.5,8.75);}
			\foreach \x in {0,...,16}
			\foreach \y in {0,...,16} {
				\pgfmathtruncatemacro{\val}{mod(\x+\y,2)};
				\pgfmathtruncatemacro{\grayshade}{20*\val};
				\draw [color=lightgray, fill=black!\grayshade] (\y*.5,\x*.5) rectangle (\y*.5+.5,\x*.5+.5);
				\ifthenelse{\val > 0}{\pgfmathtruncatemacro{\valu}{mod(\y,8)}}{\pgfmathtruncatemacro{\valu}{mod(\x,8)}};
				\node at (\y*.5+.25,\x*.5+.25) {\valu};
			}
		\end{tikzpicture}
	\end{center}
	\caption{The coloring $\theta$ for a part of the grid for the case $p=3$ is shown above. The cell at the lower left corner corner corresponds to the vertex $(0,0)$. The cell corresponding to a vertex $(i,j)\in R$ having $\theta(i,j)=(0,t)$ is colored white and contains the integer $t$; similarly, the cell corresponding to a vertex $(i,j)\in C$ having $\theta(i,j)=(1,t)$ is colored gray and contains the integer $t$.}\label{fig:theta}
\end{figure}

\subsection{Obtaining a $p$-centered coloring}
We now describe a new coloring  $\lambda$ of the grid $G$. First, we partition $V(G)$ into two sets in two different ways. The vertex set $V(G)$ is partitioned into two sets $R,C$ as follows: define $R=\{(x,y):\left \lfloor\frac{x}{3}\right\rfloor +\left \lfloor \frac{y}{3} \right \rfloor$ is even$\}$ and $C=V(G)\setminus R$. Again, $V(G)$ is partitioned into two sets $A,B$ in the following way: define  $A=\{(x,y)\in V(G):\fmod(x,3)+\fmod(y,3)$ is odd$\}$ and $B=V(G)\setminus A$. See Figure~\ref{fig:partition} for an illustration of the two partitions of $V(G)$.
The mapping $I:V(G)\rightarrow V(G)$ is defined as $(x,y)\mapsto (\left\lfloor \frac{x}{3}\right\rfloor,\left\lfloor \frac{y}{3}\right\rfloor)$.
Finally, we define the coloring $\lambda$ of $V(G)$ as the function given below:
$$\begin{array}{llll}
(x,y)&\mapsto&(\mu(I(x,y)),\fmod(x,3),\fmod(y,3))&\mbox{when }(x,y)\in A\\
(x,y)&\mapsto&(0,\fmod(x,6p+6))&\mbox{when }(x,y) \in R \cap B,\mbox{ and }\\
(x,y)&\mapsto&(1,\fmod(y,6p+6))&\mbox{when }(x,y)\in C\cap B
\end{array}$$
Notice that $I(V(G))=V(G)$ and therefore $|\mu(I(V(G)))|\leq 32p$ as observed in Section~\ref{sec:mu}. This implies that $|\lambda(V(G))|\leq 4(32p)+2(6p+6)=140p+12$.

\begin{figure}
	\begin{center}
		\begin{tikzpicture}
			\foreach \x in {0,...,17} {
				\draw [lightgray] (0,\x*.5) -- (8.75,\x*.5);
				\draw [lightgray] (\x*.5,0) -- (\x*.5,8.75);}
			\foreach \x in {0,...,16}
			\foreach \y in {0,...,16} {
				\pgfmathtruncatemacro{\val}{mod(int(\x/3)+int(\y/3),2)};
				\pgfmathtruncatemacro{\grayshade}{20*\val};
				\draw [color=lightgray, fill=black!\grayshade] (\y*.5,\x*.5) rectangle (\y*.5+.5,\x*.5+.5);
				\pgfmathtruncatemacro{\pat}{mod(mod(\x,3)+mod(\y,3),2)};
				\ifthenelse{\pat>0}{\fill [pattern=north west lines] (\y*.5,\x*.5) rectangle (\y*.5+.5,\x*.5+.5)}{};
			}
		\end{tikzpicture}
	\end{center}
	\caption{The above figure shows the two partitions of $V(G)$. The cell at position $(i,j)$ is: (a) white if $(i,j)\in R$ and gray if $(i,j)\in C$, (b) hatched if $(i,j)\in A$ and not hatched if $(i,j)\in B$.}\label{fig:partition}
\end{figure}
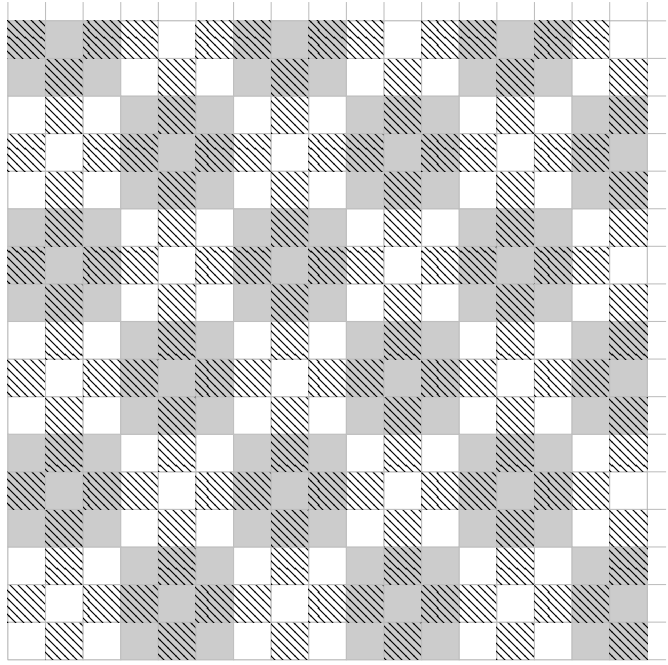

\begin{obs}\label{adjacent}
If $(x,y)(x',y')\in E(G)$ such that $I(x,y)\neq I(x',y')$,  then one of $(x,y),(x',y')$ is in $R$ and the other in $C$. 
\end{obs}
\begin{proof}
Since $(x,y)(x',y')\in E(G)$, we have $|x-x'|+|y-y'|=1$. This means that either $x=x'$ or $y=y'$. We shall only consider the case when $x=x'$ as the other case is symmetric. Note that we now have that $|y-y'|=1$. As $\Big\lfloor\frac{x}{3}\Big\rfloor=\left\lfloor\frac{x'}{3}\right\rfloor$, the fact that $I(x,y)\neq I(x',y')$ implies that $\Big\lfloor\frac{y}{3}\Big\rfloor\neq\left\lfloor\frac{y'}{3}\right\rfloor$. Since $|y-y'|=1$, we can conclude that $\Big|\Big\lfloor\frac{y}{3}\Big\rfloor-\left\lfloor\frac{y'}{3}\right\rfloor\Big|=1$. Then the parity of $\left\lfloor \frac{x'}{3} \right \rfloor+\left\lfloor \frac{y'}{3} \right \rfloor$ is different from that of $\left\lfloor \frac{x}{3} \right \rfloor+\left\lfloor \frac{y}{3} \right \rfloor$. Thus, we have that one of $(x,y),(x',y')$ is in $R$ and the other in $C$.\hfill\qed
\end{proof}
\begin{obs}\label{bridge}
Let $(x_1,y_1),(x_2,y_2)\in V(G)$ and let $P$ be a path in $G$ having endvertices $(x_1,y_1)$ and $(x_2,y_2)$. If $y_1<y_2$, then for every $b$ such that $y_1<b\leq y_2$, there exists an edge $(a,b-1)(a,b)\in E(P)$, for some $a\in\mathbb{N}$. Similarly, if $x_1<x_2$, then for every $a$ such that $x_1<a\leq x_2$, there exists an edge $(a-1,b)(a,b)\in E(P)$, for some $b\in\mathbb{N}$. 
\end{obs}
\begin{proof}
We prove only the first statement as the second one is symmetric. Suppose that $y_1<y_2$. Let $u$ be the first vertex on $P$ for which $\pi_y(u)\geq b$ when traversing it starting from the endpoint $(x_1,y_1)$, and let $v$ be the vertex encountered just before $u$ in the traversal. By our choice of $u$, we have that $\pi_y(v)<b$. Since $uv\in E(P)\subseteq E(G)$, we also have that $|\pi_y(u)-\pi_y(v)|\leq 1$, which implies that $\pi_y(v)=b-1$ and $\pi_y(u)=b$.  As $\pi_y(v)\neq\pi_y(u)$, we have that $\pi_x(v)=\pi_x(u)=a$ (say). Then $vu=(a,b-1)(a,b)\in E(P)$, as required.\hfill\qed
\end{proof}

\begin{lemma}\label{existence}
Let $H$ be a connected subgraph of $G$ such that $V(H)$ spans at least $6p+9$ rows or at least $6p+9$ columns, then $|\lambda(V(H))|\geq p+1$.
\end{lemma}
\begin{proof}
We shall only give a proof for the case when $H$ spans at least $6p+9$ columns, as the  other case is symmetric.
Let $y_1=\min\pi_y(V(H))$ and $y_2=\max\pi_y(V(H))$. Then we have $y_2-y_1\geq 6p+8$. Clearly, there exist $x_1,x_2\in\mathbb{N}$ such that $(x_1,y_1),(x_2,y_2)\in V(H)$. As $H$ is a connected subgraph of $G$, there exists a path $P$ in $H$ having endvertices $(x_1,y_1)$ and $(x_2,y_2)$.

\begin{claim}
	Let $b\in\mathbb{N}$ such that $y_1+1<b<y_2$ and $3\mid b$. Then there exists $(x,y)\in V(P)\cap C\cap B$ such that $b-2\leq y \leq b+1$.
\end{claim}

The proof of this claim is obvious once it is noted that the vertices in $C\cap B\cap (\mathbb{N}\times [b-2,b+1])$ form a zig-zag separator that separates every vertex on its left from the ones on its right, as shown in Figure~\ref{fig:zigzag}. However, we give a formal proof for the sake of completeness.
\begin{figure}
	\begin{center}
		\begin{tikzpicture}
			\foreach \x in {1,...,17}
				\draw [lightgray] (-0.1,\x*.5) -- (3.1,\x*.5);
			\foreach \x in {0,...,6}
				\draw [lightgray] (\x*.5,0.4) -- (\x*.5,8.6);
			\foreach \x in {1,...,16}
			\foreach \y in {0,...,5} {
				\pgfmathtruncatemacro{\val}{mod(int(\x/3)+int(\y/3),2)};
				\pgfmathtruncatemacro{\grayshade}{20*\val};
				\draw [color=lightgray, fill=black!\grayshade] (\y*.5,\x*.5) rectangle (\y*.5+.5,\x*.5+.5);
				\pgfmathtruncatemacro{\pat}{mod(mod(\x,3)+mod(\y,3),2)};
				\ifthenelse{\pat>0}{\fill [pattern=north west lines] (\y*.5,\x*.5) rectangle (\y*.5+.5,\x*.5+.5)}{};
			}
			\node at (1.75,8.7) {$b$};
		\end{tikzpicture}\hspace{1in}
		\begin{tikzpicture}
			\foreach \x in {1,...,17}
				\draw [lightgray] (-0.1,\x*.5) -- (3.1,\x*.5);
			\foreach \x in {0,...,6}
				\draw [lightgray] (\x*.5,0.4) -- (\x*.5,8.6);
			\foreach \x in {1,...,16}
			\foreach \y in {0,...,5} {
				\pgfmathtruncatemacro{\val}{mod(int(\x/3)+int(\y/3),2)};
				\pgfmathtruncatemacro{\grayshade}{20*\val};
				\pgfmathtruncatemacro{\pat}{mod(mod(\x,3)+mod(\y,3),2)};
				\ifthenelse{\pat=0 \and \y>0 \and \y<5}{\draw [color=lightgray, fill=black!\grayshade] (\y*.5,\x*.5) rectangle (\y*.5+.5,\x*.5+.5);}{}
			}
			\node at (1.75,8.7) {$b$};
		\end{tikzpicture}
	\end{center}
	\caption{The figure on the left shows the portion of the grid in the vicinity of column $b$, where the meaning of the colours and patterns in each cell is as in Figure~\ref{fig:partition}. The figure on the right shows the same portion of the grid, but only the vertices in $C\cap B\cap (\mathbb{N}\times [b-2,b+1])$ are shaded gray.}\label{fig:zigzag}
\end{figure}
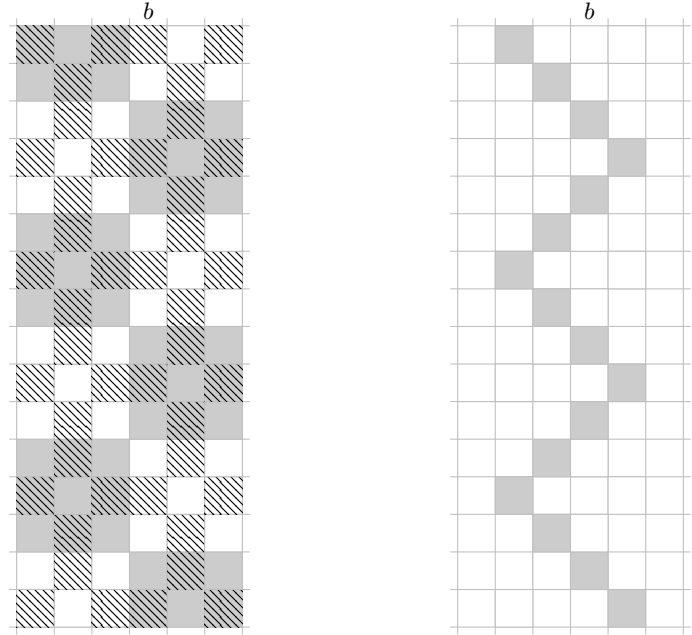

\noindent\textit{Proof of claim.} By Observation~\ref{bridge}, there exists $a\in\mathbb{N}$ such that $(a,b-1)(a,b)\in E(P)$. Using Observation~\ref{adjacent} and the fact that $3\mid b$, we have that either $(a,b-1)\in C$ or $(a,b)\in C$. Thus there exists $b'\in\{b-1,b\}$ such that $(a,b')\in C$.
If $(a,b')\in B$, then we are done as $(a,b')$ can be chosen as $(x,y)$. So we can assume that $(a,b')\notin B$. Then $\fmod(a,3)+\fmod(b',3)$ is odd. Since $\fmod(b,3)=0$, it follows that $\fmod(b',3)$ is even. This implies that $\fmod(a,3)$ is odd, which in turn means that $\fmod(a,3)=1$. Let $(x,y)$ be the neighbour of $(a,b')$ on $P$ that does not belong to $\{(a,b),(a,b-1)\}$. As $(a,b')(x,y)\in E(G)$, we have $|a-x|+|b'-y|=1$. Thus either $a=x$ and $|b'-y|=1$, or $b'=y$ and $|a-x|=1$. Suppose that $a=x$ and $|b'-y|=1$. Then as $(x,y)\notin\{(a,b-1),(a,b)\}$, we have $y=b-2$ if $b'=b-1$ and $y=b+1$ if $b'=b$. Therefore, as $3\mid b$, we have $\left\lfloor \frac{x}{3} \right \rfloor+ \left\lfloor \frac{y}{3} \right \rfloor=\left\lfloor \frac{a}{3} \right \rfloor+ \left\lfloor \frac{b'}{3} \right \rfloor$ and $|(\fmod(x,3)+\fmod(y,3))-(\fmod(a,3)+\fmod(b',3))|=1$. Since $(a,b')\in C$, we then have $(x,y)\in C$, and since $(a,b')\notin B$, we have $(x,y)\in B$, and we are done. Next, suppose that $b'=y$ and $|a-x|=1$. Then as $\fmod(a,3)=1$, we have that $\left\lfloor \frac{x}{3} \right \rfloor+ \left\lfloor \frac{y}{3} \right \rfloor=\left\lfloor \frac{a}{3} \right \rfloor+ \left\lfloor \frac{b'}{3} \right \rfloor$  and $|(\fmod(a,3)+\fmod(b',3))-(\fmod(x,3)+\fmod(y,3))|=1$. We then again get that $(x,y)\in C\cap B$.
Notice that in all cases, we find a vertex $(x,y)\in V(P)\cap C\cap B$ such that $y\in [b-2,b+1]$. This proves the claim.\medskip

Now let $Z=\{b\in [y_1+2,y_1+6p+7]\colon 6\mid b\}$. Since $y_2-y_1\geq 6p+8$, we have that $Z\subseteq [y_1+2,y_2-1]$. Then for each $b\in Z$, we know by the above claim that there exists $u_b\in V(P)\cap C\cap B$ such that $\pi_y(u_b)\in [b-2,b+1]$. Let $U=\{u_b\colon b\in Z\}$. Clearly, for distinct $b,b'\in Z$, we have $|b-b'|\geq 6$, and therefore $[b-2,b+1] \cap [b'-2,b'+1]=\emptyset$, which implies that $u_b\neq u_{b'}$. Thus $|U|=|Z|\geq\left\lfloor\frac{(y_1+6p+7)-(y_1+2)+1}{6}\right\rfloor\geq p+1$. Further, for distinct vertices $u,u'\in U$, since $\pi_y(u)\neq\pi_y(u')$ and $|\pi_y(u)-\pi_y(u')|\leq 6p+5$ , we have $\fmod(\pi_y(u),6p+6)\neq\fmod(\pi_y(u'),6p+6)$, which implies that $\lambda(u)\neq\lambda(u')$ (recall that $U\subseteq C\cap B$). Thus, $|\lambda(U)|=|U|\geq p+1$. As $U\subseteq V(P)\subseteq V(H)$, it now follows that $|\lambda(V(H))|\geq p+1$, and we are done.\hfill\qed
\end{proof}

\begin{theorem}
$\lambda$ is a $p$-centered coloring of $G$.
\end{theorem}
\begin{proof}
We shall show that with respect to the coloring $\lambda$, every connected subgraph of $G$ either contains more than $p$ colors or a vertex with a unique color.
Let $H$ be any connected subgraph of $G$. Clearly, $V(H)$ is both row and column contiguous. If $V(H)$ spans either at least $6p+9$ rows or at least $6p+9$ columns, then by Lemma~\ref{existence}, we have $|\lambda(V(H))|\geq p+1$ and we are done. So let us assume that $V(H)$ spans at most $6p+8$ rows and at most $6p+8$ columns. We will show that $H$ contains a uniquely colored vertex. Since this is trivially true when $|V(H)|=1$, we shall assume that $|V(H)|>1$.

First suppose that $V(H)\cap A=\emptyset$. It should be obvious from Figure~\ref{fig:partition} and the definition of $\lambda$ that in this case, $H$ is isomorphic to a cycle on four vertices or a path on at most three vertices, and every vertex of $H$ has a different color. Nevertheless, we give a formal proof for the sake of completeness. Consider any $(x,y)\in V(H)$. Clearly, $(x,y)\in B$. Notice that if either $\fmod(x,3)=1$ or $\fmod(y,3)=1$, then $\fmod(x,3)=\fmod(y,3)=1$ and all neighbours of $(x,y)$ in $G$ belong to $A$. Since $H$ is connected, this implies that if there exists $(x,y)\in V(H)$ such that $\fmod(x,3)=1$ or $\fmod(y,3)=1$, then $|V(H)|=1$, contradicting our assumption that $|V(H)|>1$. So we assume that $V(H)\subseteq\{(x,y)\in V(G)\colon\fmod(x,3),\fmod(y,3)\in\{0,2\}\}$. Since $V(H)$ is both row and column contiguous, this implies that $|\pi_x(V(H))|,|\pi_y(V(H))|\leq 2$. Suppose that there are distinct vertices $u,v\in V(H)$ such that $\lambda(u)=\lambda(v)$. Then as $V(H)\subseteq B$, we have that either $u,v\in R$ or $u,v\in C$. Since the cases are symmetric, we assume that $u,v\in R$. Then $\lambda(u)=\lambda(v)$ implies that $\fmod(\pi_x(u),6p+6)=\fmod(\pi_x(v),6p+6)$.  As $|\pi_x(V(H))|\leq 2$, this further implies that $\pi_x(u)=\pi_x(v)$. Then since $|\pi_y(V(H))|\leq 2$ and $u\neq v$, we have $|\pi_y(u)-\pi_y(v)|=1$. Assuming without loss of generality that $\pi_y(v)>\pi_y(u)$, we have $\pi_y(v)=\pi_y(u)+1$. Recalling that $\fmod(\pi_y(u),3),\fmod(\pi_y(v),3)\in\{0,2\}$,  we then have $\pi_y(u)=3k-1$ and $\pi_y(v)=3k$, for some $k\in\mathbb{N}$. Thus $\left\lfloor\frac{\pi_y(v)}{3}\right\rfloor=\left\lfloor\frac{\pi_y(u)}{3}\right\rfloor+1$. Since $\left\lfloor\frac{\pi_x(u)}{3}\right\rfloor=\left\lfloor\frac{\pi_x(v)}{3}\right\rfloor$, we now have that $\left\lfloor\frac{\pi_x(u)}{3}\right\rfloor+\left\lfloor\frac{\pi_y(u)}{3}\right\rfloor$ and $\left\lfloor\frac{\pi_x(v)}{3}\right\rfloor+\left\lfloor\frac{\pi_y(v)}{3}\right\rfloor$ are of different parity. This contradicts the fact that $u,v\in R$. Thus, there exist no two vertices in $V(H)$ that have the same color with respect to the coloring $\lambda$, or in other words, every vertex in $H$ has a unique color, and we are done.

Next, suppose that $V(H)\cap A \neq \emptyset$.
We claim that $I(V(H)\cap A)$ is row and column contiguous. This is also not difficult to see: for example, it becomes clear that $I(V(H)\cap A)$ is column contiguous if one observes that for any $i\in\mathbb{N}$ such that $i\notin\pi_y(I(V(H)\cap A))$, the connected subgraph $H$ has to lie entirely on one side of the separator $I^{-1}(\mathbb{N}\times\{i\})\cap A$ (see Figure~\ref{fig:colcont}), which implies that either $i<\min\pi_y(I(V(H)\cap A))$ or $i>\max\pi_y(I(V(H)\cap A))$.
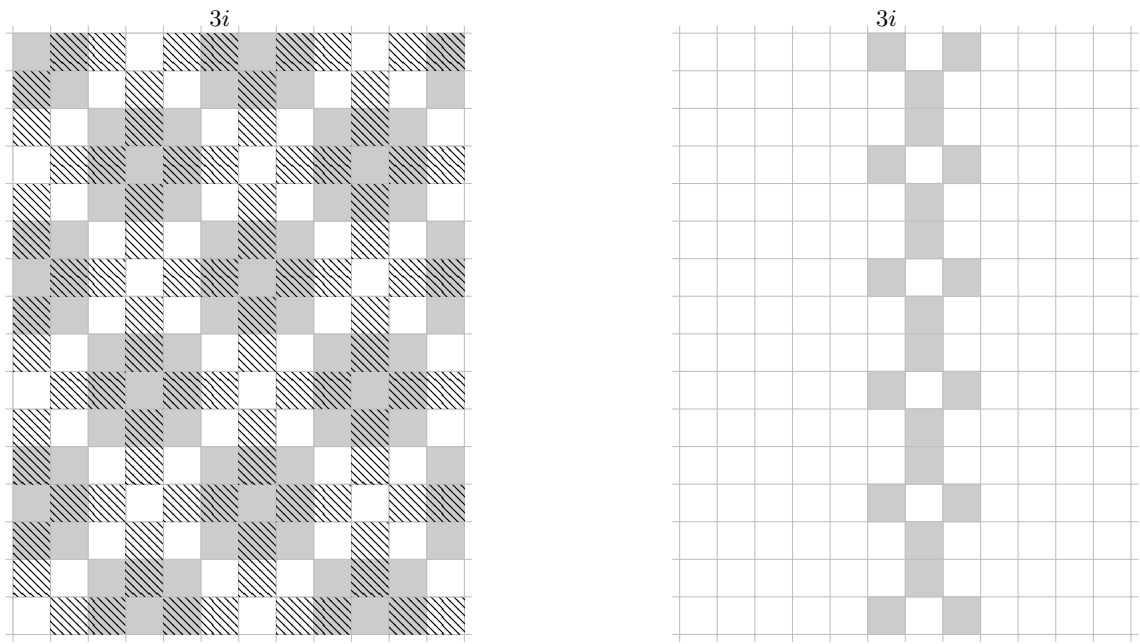
\begin{figure}
	\begin{center}
		\begin{tikzpicture}
			\foreach \x in {1,...,17}
				\draw [lightgray] (0.4,\x*.5) -- (6.6,\x*.5);
			\foreach \x in {1,...,13}
				\draw [lightgray] (\x*.5,0.4) -- (\x*.5,8.6);
			\foreach \x in {1,...,16}
			\foreach \y in {1,...,12} {
				\pgfmathtruncatemacro{\val}{mod(int(\x/3)+int(\y/3),2)};
				\pgfmathtruncatemacro{\grayshade}{20*\val};
				\draw [color=lightgray, fill=black!\grayshade] (\y*.5,\x*.5) rectangle (\y*.5+.5,\x*.5+.5);
				\pgfmathtruncatemacro{\pat}{mod(mod(\x,3)+mod(\y,3),2)};
				\ifthenelse{\pat>0}{\fill [pattern=north west lines] (\y*.5,\x*.5) rectangle (\y*.5+.5,\x*.5+.5)}{};
			}
			\node at (3.25,8.7) {$3i$};
		\end{tikzpicture}\hspace{1in}
		\begin{tikzpicture}
			\foreach \x in {1,...,17}
				\draw [lightgray] (0.4,\x*.5) -- (6.6,\x*.5);
			\foreach \x in {1,...,13}
				\draw [lightgray] (\x*.5,0.4) -- (\x*.5,8.6);
			\foreach \x in {1,...,16}
			\foreach \y in {1,...,12} {
				\pgfmathtruncatemacro{\val}{mod(int(\x/3)+int(\y/3),2)};
				\pgfmathtruncatemacro{\pat}{mod(mod(\x,3)+mod(\y,3),2)};
				\ifthenelse{\pat>0 \and \y>5 \and \y<9}{\draw [color=lightgray, fill=black!20] (\y*.5,\x*.5) rectangle (\y*.5+.5,\x*.5+.5);}{}
			}
			\node at (3.25,8.7) {$3i$};
		\end{tikzpicture}
	\end{center}
	\caption{The figure on the left shows a portion of the grid containing the column $3i$, where $i\in\mathbb{N}$; the colours and patterns in each cell depict the sets $R$, $C$, $A$ and $B$ as in Figure~\ref{fig:partition}. The figure on the right shows the same portion of the grid, but only the vertices in $I^{-1}(\mathbb{N}\times\{i\})\cap A$ are shaded gray.}\label{fig:colcont}
\end{figure}
Anyway, we give below a formal proof for the fact that $I(V(H)\cap A)$ is column contiguous.
Consider any $y,y'\in\pi_y(I(V(H)\cap A))$ and $b\in\mathbb{N}$ such that $y<b<y'$.
Then there exists $b_1\in\{3y,3y+1,3y+2\}$ such that $b_1\in\pi_y(V(H))$ and $b_2\in\{3y',3y'+1,3y'+2\}$ such that $b_2\in\pi_y(V(H))$. Since $b_1<3b+1<b_2$ and $V(H)$ is a column contiguous set, we have that $3b+1\in\pi_y(V(H))$. Thus, there exists $a\in\mathbb{N}$ such that $(a,3b+1)\in V(H)$.
If $(a,3b+1)\notin A$ then $\fmod(a,3)=1$, which implies that all the neighbours of $(a,3b+1)$ are in $A$. Since $(a,3b+1)\in V(H)$, $H$ is connected and $|V(H)|>1$, some neighbour $u$ of $(a,3b+1)$ is in $V(H)$. Then we have $u\in V(H)\cap A$ and $\left\lfloor\frac{\pi_y(u)}{3} \right \rfloor=b$, which implies that $b\in\pi_y(I(V(H)\cap A))$. This proves that $I(V(H)\cap A)$ is column contiguous. Symmetrically, we can prove that $I(V(H)\cap A) $ is row contiguous.
Since $V(H)$ spans at most $6p+8$ rows and at most $6p+8$ columns, it follows that $I(V(H)\cap A)$ spans at most $2p+4$ rows and at most $2p+4$ columns. As $p>1$, we have that $2p+4\leq 4p$, and therefore we can apply Lemma~\ref{small} to the set $I(V(H)\cap A)$ to conclude that there is a uniquely colored vertex $w\in I(V(H)\cap A)$ with respect to the coloring $\mu$. Since $w\in I(V(H)\cap A)$, we know that $I^{-1}(w)\cap A \cap V(H)\neq\emptyset$. Let $u\in I^{-1}(w)\cap A\cap V(H)$. Suppose for the sake of contradiction that there exists $v\in V(H)$ such that $v\neq u$ and $\lambda(v)=\lambda(u)$. Then clearly, $v\in A$ and $\mu(I(v))=\mu(I(u))=\mu(w)$. Since $w=I(u)$ is a uniquely colored vertex in $I(V(H)\cap A)$ with respect to the coloring $\mu$, and $I(v)\in I(V(H)\cap A)$, we then have that $I(v)=I(u)=w$. Now observe that from the definition of $\lambda$, no two vertices in $I^{-1}(w)\cap A$ have the same color with respect to the coloring $\lambda$. It follows that $\lambda(u)\neq\lambda(v)$, which contradicts the assumption that $\lambda(v)=\lambda(u)$. This proves that $u$ is a uniquely colored vertex in $H$ with respect to the coloring $\lambda$.
\hfill\qed
\end{proof}
\subsection{Generalizing to all values of $p$}
It is quite easy to see that the above construction can be used to get a $p$-centered coloring of $G$ for arbitrary values of $p$ as follows.
Let $p$ be any positive integer. If $p=1$, then any proper coloring of the grid is a $p$-centered coloring, and hence we have that there is a $p$-centered coloring of the grid using just 2 colors (for $p=1$, the coloring $\lambda$, if constructed overlooking the fact that the condition that $p$ has to be at least 2 is not met, is also a proper coloring of $G$ and hence is a $p$-centered coloring of $G$). So we assume that $p>1$. Then let $p'$ be the smallest power of 2 that is greater than or equal to $p$. Clearly, $p'\geq 2$ and $p'\leq 2p$. Using the construction explained above, we can get a $p'$-centered coloring $\lambda$ of $G$ using at most $140p'+12$ colors. Notice that any $p'$-centered coloring of $G$ is also a $p$-centered coloring of $G$ as $p\leq p'$. Recalling that $p'\leq 2p$, we now have that $\lambda$ is a $p$-centered coloring of $G$ using at most $140p'+12\leq 280p+12$ colors. Thus $G$ has a $p$-centered coloring using $O(p)$ colors for every positive integer $p$.
\section{Conclusion}
We hope that our strategy for constructing a $p$-centered coloring for the planar grid can serve as a starting point for research on explicit constructions of $p$-centered colorings using $O(p)$ colors for more general classes of graphs with bounded maximum degree. A natural next step could be to try if similar strategies work for bounded degree graphs that are similar to the planar grid --- for example, the planar grid with both diagonals added in each square, or the three dimensional grid.

\bibliographystyle{plain}
\bibliography{centered}

\begin{thebibliography}{1}

\bibitem{Bodlaenderetal}
Hans~L. Bodlaender, John~R. Gilbert, Hj\'almt\'yr Hafsteinsson, and Ton Kloks.
\newblock Approximating treewidth, pathwidth, frontsize, and shortest
  elimination tree.
\newblock {\em Journal of Algorithms}, 18(2):238--255, 1995.

\bibitem{DeVosetal}
Matt DeVos, Guoli Ding, Bogdan Oporowski, Daniel~P. Sanders, Bruce Reed, Paul
  Seymour, and Dirk Vertigan.
\newblock Excluding any graph as a minor allows a low tree-width 2-coloring.
\newblock {\em Journal of Combinatorial Theory, Series B}, 91(1):25--41, 2004.

\bibitem{diestel}
Reinhard Diestel.
\newblock {\em Graph Theory (Graduate Texts in Mathematics)}.
\newblock Springer, August 2005.

\bibitem{DFMS}
Micha\l\ D\k{e}bski, Stefan Felsner, Piotr Micek, and Felix Schr\"oder.
\newblock Improved bounds for centered colorings.
\newblock {\em Advances in Combinatorics}, 2021:8, 28pp.

\bibitem{NesPOMtreedepth}
Jaroslav Ne\v{s}et\v{r}il and Patrice {Ossona de Mendez}.
\newblock Tree-depth, subgraph coloring and homomorphism bounds.
\newblock {\em European Journal of Combinatorics}, 27(6):1022--1041, 2006.

\bibitem{NesPOMexpansion}
Jaroslav Ne\v{s}et\v{r}il and Patrice {Ossona de Mendez}.
\newblock Grad and classes with bounded expansion {I}. {D}ecompositions.
\newblock {\em European Journal of Combinatorics}, 29(3):760--776, 2008.

\bibitem{PiliSieb}
Micha\l\ Pilipczuk and Sebastian Siebertz.
\newblock Polynomial bounds for centered colorings on proper minor-closed graph
  classes.
\newblock {\em Journal of Combinatorial Theory, Series B}, 151:111--147, 2021.

\end{thebibliography}
\end{document}